\documentclass[11pt,a4paper]{amsart}
\usepackage{amssymb,xspace}
\usepackage{amstext}
\usepackage{tikz}
\usetikzlibrary{arrows,decorations.pathmorphing,backgrounds,fit,positioning,shapes.symbols,chains}
\theoremstyle{plain}
\usepackage{amsbsy,amssymb,amsfonts,latexsym}

\usepackage{enumerate}
\usepackage{graphics}

\date{\today}
\title[]{Two problems on weighted shifts in linear dynamics}
\author{Fr\'{e}d\'{e}ric Bayart}
\address{Laboratoire de Mathématiques Blaise Pascal UMR 6620 CNRS, Université Clermont Auvergne, Campus universitaire des Cézeaux, 3 place Vasarely, 63178 Aubière Cedex, France.} 
\email{frederic.bayart@uca.fr}
\thanks{The author was partially supported by the grant ANR-17-CE40-0021 of the French National Research Agency ANR (project Front).}



\newcommand{\veps}{\varepsilon}

\def\RR{\mathbb R}

\def\NN{\mathbb N}
\def\ZZ{\mathbb Z}

\def\DD{\mathbb D}

\def\CC{\mathbb C}

\def\card{\textrm{card}}

\DeclareMathOperator{\vect}{span}

\newcommand{\Lip}{\mathrm{Lip}}

\newtheorem{theorem}{Theorem}[section]

\newtheorem{lemma}[theorem]{Lemma}

\newtheorem{corollary}[theorem]{Corollary}

{\theoremstyle{definition}}
{\theoremstyle{definition}}

{\theoremstyle{definition}}

{\theoremstyle{definition}}

{\theoremstyle{definition}}

{\theoremstyle{definition}}

\newtheorem*{THMA}{Theorem (Bernardes, Messaoudi)}

\begin{document}

\begin{abstract}
We show that an invertible bilateral weighted shift is strongly structurally stable if and only if it has the shadowing property. We also exhibit a Köthe sequence space supporting a frequently hypercyclic weighted shift, but no chaotic weighted shifts.
\end{abstract}

\maketitle

\section{Introduction}

Linear dynamics began in the 1980's with the thesis of Kitai \cite{Kit82} and the paper of Gethner and Shapiro \cite{GeSh87}. Its aim
is to study the dynamical properties of a (bounded) operator $T$ acting on some (complete) linear space $X$,
as one studies more usually the dynamics of a continuous map acting on a (compact) metric space $X$.

At its beginning, linear dynamics was essentially devoted to the study of the density of orbits, thanks to its link with the invariant subspace/subset problem
leading to the notions of hypercyclicity, supercyclicity and their variants. More recently, various notions arising in classical dynamics have also been investigated
in the linear context, for instance rigidity, the specification property, the shadowing property, and so on...

One of the interests of linear dynamics is that we have a lot of examples of operators at our disposal to illustrate the definitions and to provide examples and counterexamples.
Among these examples, the most studied class is certainly that of weighted shifts: given a sequence $(w_n)_n$ of positive real numbers,
the corresponding weighted shifts are formally defined by
\begin{align*}
 B_w(x_n)_{n\geq 0}&=(w_{n+1}x_{n+1})_{n\geq 0}\textrm{ (unilateral weighted shifts)}\\
 B_w(x_n)_{n\in\ZZ}&=(w_{n+1}x_{n+1})_{n\in\ZZ}\textrm{ (bilateral weighted shifts)}.
\end{align*}
From the characterization of Salas \cite{Sal95} of hypercyclic weighted shifts on $\ell_p$ or $c_0$, their dynamical properties have been widely studied. It turns out that,
very recently, they appear in two problems where open questions still exist.

\subsection{Hyperbolicity and strong structural stability}
An operator $T$ on a complex Banach space $X$ is said to be {\bf hyperbolic} if its spectrum $\sigma(T)$ does not intersect the unit circle.
It is said {\bf strongly structurally stable} if for every $\veps>0$, there exists $\delta>0$ such that, for any Lipschitz map $\alpha\in \textrm{Lip}(X)$,
with $\|\alpha\|_\infty\leq\delta$ and $\textrm{Lip}(\alpha)\leq\delta$, there is a homeomorphism $\varphi:X\to X$ with $\|\varphi\|_\infty\leq\veps$
such that $T\circ (I+\varphi)=(I+\varphi)\circ (T+\alpha).$ Namely, a small perturbation $T+\alpha$ of $T$ is conjugated to $T$ via a homeomorphism
close to the identity operator. It is known from the 1960's that every hyperbolic operator is strongly structurally stable
(see \cite{Hartman63,Palis68,Pugh69} and for a recent advance \cite{BerMes20a}). That the converse is false was shown only very recently in \cite{BerMes20} where the following theorem is obtained.

\begin{THMA}
 Let $X=\ell_p(\ZZ)$, $p\in[1,+\infty)$ or $X=c_0(\ZZ)$. Let $w=(w_n)_{n\in\ZZ}$ be a bounded and bounded below sequence of positive integers and let $B_w$ be the associated
 weighted shift. If
 \[ \lim_{n\to+\infty} \sup_{k\in\NN} (w_{-k}w_{-k-1}\cdots w_{-k-n})^{1/n}<1\textrm{ and } \lim_{n\to+\infty}\inf_{k\in\NN}(w_kw_{k+1}\cdots w_{k+n})^{1/n}>1\]
 then $B_w$ is strongly structurally stable and not hyperbolic.
\end{THMA}

This result leaves open the characterization of the strongly structurally stable weighted shifts on $\ell_p(\ZZ)$ or $c_0(\ZZ)$. We provide
such a characterization.

\begin{theorem}\label{thm:sss}
 Let $X=\ell_p(\ZZ)$, $p\in[1,+\infty)$ or $X=c_0(\ZZ)$. Let $w=(w_n)_{n\in\ZZ}$ be a bounded and bounded below sequence of positive integers and let $B_w$ be the associated
 weighted shift. Then $B_w$ is strongly structurally stable if and only if one the following conditions holds:
 \begin{enumerate}[(A)]
  \item $\lim_{n\to+\infty} \sup_{k\in\ZZ} (w_k\cdots w_{k+n})^{1/n}<1$;
  \item $\lim_{n\to+\infty} \inf_{k\in\ZZ} (w_k\cdots w_{k+n})^{1/n}>1$;
  \item $\lim_{n\to+\infty} \sup_{k\in\NN} (w_{-k}\cdots w_{-k-n})^{1/n}<1$ and $\lim_{n\to+\infty} \inf_{k\in \NN}(w_k\cdots w_{k+n})^{1/n}>1$.
 \end{enumerate}
\end{theorem}

This characterization should be compared with the characterization of weighted shifts having the shadowing property, a property of dynamical systems that we recall now.
A sequence $(x_n)_{n\in\ZZ}$ in $X$ is called a $\delta$-pseudotrajectory of $T\in\mathfrak L(X)$, with $\delta>0$, if $\|Tx_n-x_{n+1}\|\leq\delta$ for all $n\in\ZZ$. 
An invertible operator $T$ is said to have the {\bf shadowing property} if for every $\veps>0$, there exists $\delta>0$ such that every $\delta$-pseudotrajectory is $\veps$-shadowed
by a real trajectory, namely there exists $x\in X$ such that
\[ \|T^n x-x_n\|<\veps\textrm{ for all }n\in\ZZ. \]
Comparing Theorem \ref{thm:sss} and \cite[Theorem 18]{BerMes20}, we get the following corollary:
\begin{corollary}
 Let $B_w$ be an invertible weighted shift acting on $\ell_p(\ZZ)$, $p\in[1,+\infty)$ or $c_0(\ZZ)$.
 Then $B_w$ is strongly structurally stable if and only if $B_w$ has the shadowing property.
\end{corollary}

It is unknown whereas strong structural stability implies the shadowing property or if the converse holds. Unfortunately, 
weighted shifts on $\ell_p$ or $c_0$ cannot provide a counterexample.

\subsection{Frequently hypercyclic and chaotic weighted shifts}
An operator $T$ on a separable Fréchet space $X$ is called {\bf chaotic} if it admits a dense orbit and if it has a dense set of periodic points.
It is said {\bf frequently hypercyclic} if there exists a vector $x\in X$, called a frequently hypercyclic vector for $T$, 
such that, for any nonempty open subset $U$ of $X$,
$$\underline{\mathrm{dens}}\{n\geq 0:T^n x\in U\}>0$$
where, for $A\subset\NN_0$, $\underline{\mathrm{dens}}(A)=\liminf_{N\to+\infty}\frac 1{N+1}\card\{n\leq N: n\in A\}$ denotes the lower
density of $A$.

The relationship between chaos and frequent hypercyclicity has been the subject of many investigations. Bayart and Grivaux
in \cite{BAYGRIPLMS} have built a frequently hypercyclic operator which is not chaotic; their counterexample is a unilateral backward shift on $c_0$.
In \cite{Men17}, Menet exhibited operators on any $\ell_p$ or on $c_0$ that are chaotic and not frequently hypercyclic.

Let us come back to (unilateral) weighted shifts. A natural context to study them is that of K\"othe sequence spaces. Let $A=(a_{m,n})_{m\geq 1,n\geq 0}$
be a matrix of strictly positive numbers such that, for all $m\geq 1$, $n\geq 0$, $a_{m,n}\leq a_{m+1,n}$. The K\"othe sequence space of order $p$
is defined as 
$$\lambda^p(A)=\left\{x=(x_n)_{n\geq 0}:\ \forall m\geq 1,\ \sum_{n\geq 0}|x_n|^p a_{m,n}<+\infty\right\}$$
while the K\"othe sequence space of order $0$ is given by
$$c_0(A)=\left\{x=(x_n)_{n\geq 0}:\ \forall m\geq 1,\ \lim_{n\to+\infty}|x_n|a_{m,n}=0\right\}.$$

Chaotic and frequently hypercyclic weighted shifts on K\"othe sequence spaces are studied in depth in \cite{CGM21}.
Whereas it is known for a long time that every chaotic weighted shift is frequently hypercyclic (compare \cite{Gro00a} and \cite{BAYGRITAMS}),
it is pointed out in \cite{CGM21} that there exist
\begin{itemize}
 \item K\"othe sequence spaces that do not support any chaotic or frequently hypercyclic weighted shifts (\cite[Example 3.8]{CGM21});
 \item K\"othe sequence spaces where every frequently hypercyclic weighted shift is chaotic and that support such shifts ($\ell_p$ or $H(\DD)$,
 see \cite{BAYRUZSA} and \cite[Proposition 3.12]{CGM21});
 \item K\"othe sequence spaces that support a chaotic weighted shift and a frequently hypercyclic non chaotic weighted shift ($c_0$ or $H(\CC)$,
 see \cite{BAYGRIPLMS} and \cite[Proposition 3.15]{CGM21});
\end{itemize}

Thus there remains an intriguing case \cite[Question 3.17]{CGM21}: is there a K\"othe sequence space supporting frequently hypercyclic 
weighted shifts but no chaotic weighted shifts? We will answer this question in Section 3:

\begin{theorem}\label{thm:kothe}
There exists a Köthe sequence space supporting frequently hypercyclic weighted shifts but no chaotic weighted shifts. 
\end{theorem}

%
%
%

\section{Strong structural stability of weighted shifts}

We begin the proof of Theorem \ref{thm:sss} by the following technical lemma. We fix $X=\ell_p(\ZZ)$, $p\in[1,+\infty)$, or $X=c_0(\ZZ)$ and $w=(w_n)_{n\in\ZZ}$ 
a bounded and bounded below sequence of positive integers.
\begin{lemma}\label{lem:sss}
 Let $\delta>0$, $m\in\NN$, $t\in\NN_0$. There exist $\alpha\in\mathcal \Lip(X)$ and $u\in X$ satisfying the following properties:
 \begin{itemize}
  \item $\|\alpha\|_\infty \leq\delta$;
  \item $\Lip(\alpha)\leq \delta$;
  \item $\alpha_{k+t}\big((B_w+\alpha)^{m-k-1} u\big)=\delta$ for all $k=0,\dots,m-1$.
 \end{itemize}
\end{lemma}
\begin{proof}
To simplify the notations we first prove the case $t=0$.
 For $k=0,\dots,m-1$, we set $a_{m,k}=w_{2m}w_{2m-1}\cdots w_{2m-k+1}$. We fix $\kappa>0$ so large that, for all $j\neq k$ in $\{0,\dots,m-1\}$,
 \[ \left\|\kappa a_{m,k} e_{2m-k}-\kappa a_{m,j} e_{2m-j}\right\|\geq 4. \]
 We also consider the function $\rho:\RR\to [0,1]$ defined by
 \begin{itemize}
  \item $\rho=0$ on $(-\infty,-1)$ and on $(1,+\infty)$;
  \item $\rho(0)=1$ and $\rho$ is affine on $[-1,0]$ and on $[0,1]$.
 \end{itemize}
Clearly, $\Lip(\rho)\leq 1$. We then construct by induction on $j=0,\dots,m-1$ maps $\alpha^{(j)}\in\mathcal \Lip(X)$
and vectors $y^{(j)}\in X$ so that, for all $j=0,\dots,m-1$,
\begin{enumerate}[(a)]
 \item for all $k\leq j$, $\big(B_w+\alpha^{(j)}\big)^k (\kappa e_{2m})=\kappa a_{m,k}e_{2m-k}+y^{(k)}$;
 \item $\alpha_l^{(j)}=0$ provided $l\notin \{m-1-j,\dots,m-1\}$;
 \item for all $k\in \{0,\dots,j\}$, for all $x\in X$, 
 \[ \alpha_{m-1-k}^{(j)}(x)=\delta\rho\left(\|x-\kappa a_{m,k}e_{2m-k}-y^{(k)}\|\right). \]
 \item $y^{(j)}\in \vect(e_{m-j},\dots,e_{m})$.
\end{enumerate}
Let us proceed with the construction. For $j=0$, we simply set $y^{(0)}=0$, $\alpha_l^{(0)}=0$ for $l\neq m-1$ and
\[\alpha^{(0)}_{m-1}(x)=\delta\rho\left(\|x-\kappa e_{2m}\|\right).\]
Assume that the construction has been done until $j\leq m-2$ and let us proceed with step $j+1$. We first define $y^{(j+1)}$. 
We know that
\begin{align*}
 \big( B_w+\alpha^{(j)}\big)^{j+1}(\kappa e_{2m})&=\big(B_w+\alpha^{(j)}\big) (\kappa a_{m,j}e_{2m-j}+y^{(j)}\big)\\
 &=\kappa a_{m,j+1}e_{2m-(j+1)}+\alpha^{(j)}\big(\kappa a_{m,j}e_{2m-j}+y^{(j)}\big)+B_w\big(y^{(j)}\big).
\end{align*}
This leads us to set 
\begin{equation}\label{eq:lsss1}
 y^{(j+1)}=\alpha^{(j)}\big(\kappa a_{m,j}e_{2m-j}+y^{(j)}\big)+B_w\big(y^{(j)}\big). 
\end{equation}
Since $y^{(j)}\in \vect(e_{m-j},\dots,e_{m})$ and $\alpha^{(j)}(X)\subset \vect(e_{m-1-j},\dots,e_{m})$, 
we indeed get that $y^{(j+1)}\in\vect(e_{m-(j+1)},\dots,e_{m})$. We then observe that (b) and (c) define uniquely
$\alpha^{(j+1)}$ and that $\alpha_l^{(j+1)}=\alpha_l^{(j)}$ for $l\neq m-1-(j+1)$. An important point of these definitions is that, for all $0\leq k\leq j\leq m-2$, 

\begin{equation}\label{eq:lsss2}
\alpha^{(j)}\big(\kappa a_{m,k} e_{2m-k}+y^{(k)}\big)=\alpha^{(j+1)}\big(\kappa a_{m,k} e_{2m-k}+y^{(k)}\big).
\end{equation}

Indeed, $\alpha_l^{(j+1)}=\alpha_l^{(j)}$ except for $l=(m-1)-(j+1)$ and for this value of $l$, 
$$\alpha_{m-1-(j+1)}^{(j)}=0$$
whereas
$$\alpha_{m-1-(j+1)}^{(j+1)}(\kappa a_{m,k}e_{2m-k}+y^{(k)})=\delta\rho\left(\|\kappa a_{m,k}e_{2m-k}+y^{(k)}-\kappa a_{m,j+1}e_{2m-(j+1)}-y^{(j+1)}\|\right)$$
and this is equal to $0$ by the definition of $\rho$, that of $\kappa$, and Property (d). 

Thus the construction is done and we already obtained that (b), (c) and (d) are true. It remains to prove that (a) is equally satisfied. We first observe that   \eqref{eq:lsss1} and \eqref{eq:lsss2} easily imply, by an induction on $j$, that, for all $0\leq k\leq m-2$ and all $k\leq j\leq m-1$, 

\begin{equation}\label{eq:lsss3}
y^{(k+1)}=\alpha^{(j)}\big(\kappa a_{m,k}e_{2m-k}+y^{(k)}\big)+B_w\big(y^{(k)}\big). 
\end{equation}

We fix $j=0,\dots,m-1$ and we prove (a) by induction on $k=0,\dots,j$, since the result is true for $k=0$. We write
\begin{align*}
 \big( B_w+\alpha^{(j)}\big)^{k+1}(\kappa e_{2m})&=\big(B_w+\alpha^{(j)}\big)\big(\kappa a_{m,k}e_{2m-k}+y^{(k)}\big)&\ \textrm{(induction hypothesis)}\\
&= \kappa a_{m,k+1}e_{2m-(k+1)}+y^{(k+1)}&\ \textrm{(by \eqref{eq:lsss3})}.
\end{align*}

\smallskip

We finally set $\alpha=\alpha^{(m-1)}$, $u=\kappa e_{2m}$ and we prove that the couple $(\alpha,u)$ satisfies the conclusions of Lemma \ref{lem:sss}.
Observe first that each $\alpha_k$ vanishes outside the ball $B(\kappa a_{m,m-1-k}e_{m+k+1}-y^{(m-1-k)},1)$,
$0\leq k\leq m-1$, and that these balls are pairwise disjoint. Since $\|\alpha_k\|_\infty\leq\delta$ for all $k$, we immediately get $\|\alpha\|_\infty\leq\delta$. For the proof that $\Lip(\alpha)\leq \delta$, we pick $x,y\in X$. If $\|x-y\|\geq 2$,
then $\|\alpha(x)-\alpha(y)\|\leq2\delta=\delta\|x-y\|$. Suppose now that $\|x-y\|<2$ and let us distinguish two cases:
\begin{itemize}
 \item either $x$ or $y$ belongs to some ball $B(\kappa a_{m,m-1-k}e_{m+k+1}-y^{(m-1-k)},1)$. Observe that $x$ and $y$ can belong to the same
 ball  $B(\kappa a_{m,m-1-k}e_{m+k+1}-y^{(m-1-k)},1)$ but that they cannot belong to two different balls because of they are disjoint.
It follows that
\begin{align*}
 \|\alpha(x)-\alpha(y)\|&=|\alpha_k(x)-\alpha_k(y)|\\
 &\leq \delta\big|\rho(\|\kappa a_{m,m-1-k}e_{m+k+1}-y^{(m-1-k)}-x\|)\\
 &\quad\quad\quad\quad-\rho(\|\kappa a_{m,m-1-k}e_{m+k+1}-y^{(m-1-k)}-y\|)\big|\\
 &\leq\delta\|x-y\|.
\end{align*}
\item neither $x$ nor $y$ belongs to some ball $B(\kappa a_{m,m-1-k}e_{m+k+1}-y^{(m-1-k)},1)$. In that case, $\alpha(x)=\alpha(y)=0$.
\end{itemize}
Finally the last requirement $\alpha_k \big( (B_w+\alpha)^{m-k-1}u\big)=\delta$ for all $k=0,\dots,m-1$ follows by combining properties (a) and (c) for $j=m-1$.
The proof for the other values of $t$ is similar, but we have to shift everything starting from $e_{2m+t}$ instead of $e_{2m}$.
\end{proof}

We will also need the following results coming from \cite[Lemma 19 and Proposition 15]{BerMes20}.

\begin{lemma}\label{lem:sss2}
Let $(x_n)_{n\in\mathbb N}$ be a bounded sequence of positive real numbers. Then the following assertions are equivalent:
\begin{enumerate}[(i)]
\item $\lim_{n\to+\infty}\sup_{k\in\NN} (x_k x_{k+1}\cdots x_{k+n})^{1/n}<1$;
\item $\sup_{t\in\NN}\sum_{k=0}^{+\infty} x_{1+t}x_{2+t}\cdots x_{k+t}<+\infty$;
\item $\sup_{m\in\NN}\sum_{k=0}^{m-1} x_{m-k}x_{m-k+1}\cdots x_{m}<+\infty$.
\end{enumerate}
\end{lemma}

\begin{lemma}\label{lem:sss3}
Let $w=(w_n)_{n\in\mathbb Z}$ be a bounded and bounded below sequence of positive real numbers. Assume that 
\[ \lim_{n\to+\infty} \sup_{k\in\NN} (w_{k}\cdots w_{k+n})^{1/n}<1\textrm{ and }
\lim_{n\to+\infty}\inf_{k\in \NN} (w_{-k}\cdots w_{-k-n})^{1/n}>1. \]
Then $B_w$ is not strongly structurally stable.
\end{lemma}

\begin{proof}[Proof of Theorem \ref{thm:sss}]
That a weighted shift satisfying (A), (B) or (C) is strongly structurally stable is already done in \cite{BerMes20} (see also \cite{BerMes20a}): (A) and (B) implies that $B_w$ is hyperbolic, whereas (C) implies that $B_w$ is generalized hyperbolic and a hyperbolic or generalized hyperbolic operator is always strongly structurally stable. 

Thus we start with a strongly structurally stable invertible weighted shift $B_w$ and we intend to show that one of the following condition is satisfied:
\begin{equation}\label{eq:sss1}
\lim_{n\to+\infty}\sup_{k\in\NN} (w_k\cdots w_{k+n})^{1/n}<1;
\end{equation}
\begin{equation}\label{eq:sss2}
\lim_{n\to+\infty}\inf_{k\in\NN} (w_k\cdots w_{k+n})^{1/n}>1;
\end{equation}
Assume that \eqref{eq:sss2} is not satisfied, namely that
\[\lim_{n\to+\infty} \sup_{k\in\NN} \left(\frac 1{w_k\cdots w_{k+n}}\right)^{1/n}\geq 1. \]
By Lemma \ref{lem:sss2} applied with $x_k=1/w_k$, we know that
\[ \sup_{m\in\NN} \left(\frac 1{w_1\cdots w_m}+\frac{1}{w_2\cdots w_m}+\cdots+\frac 1{w_m}\right)=+\infty. \]
Let $t\in\NN_0$. Snce $w$ is bounded and bounded below, the previous property also implies that
\[ \sup_{m\in\NN} \left(\frac 1{w_{1+t}\cdots w_{m+t}}+\frac{1}{w_{2+t}\cdots w_{m+t}}+\cdots+\frac 1{w_{m+t}}\right)=+\infty. \]
Let $\delta>0$ corresponding to $\veps=1$ in the definition of a strongly structurally stable operator. There exists $m\in\NN$ as large as we want such that
\begin{equation}
\sum_{k=0}^{m-1} \frac{1}{w_{k+1+t}\cdots w_{m+t}}\geq \frac{1+\delta}{\delta^2} \label{eq:sss3}.
\end{equation}
Let $(\alpha,u)\in \Lip(X)\times X$ given by Lemma \ref{lem:sss} for these values of $\delta,m,t$.
There exists $\varphi:X\to X$ a homeomorphism with $\|\varphi\|_\infty\leq 1$ such that 
$$(I+\varphi)\circ (B_w+\alpha)=B_w\circ (I+\varphi)$$
which gives 
$$\varphi\circ (B_w+\alpha)=-\alpha+B_w\circ\varphi$$
which in turn yields for all $k\in\ZZ$
$$\varphi_k \circ (B_w+\alpha)=-\alpha_k+w_{k+1}\varphi_{k+1}.$$
By an easy induction we then get that for all $n\in\mathbb N$, 
\[ \varphi_{n+t}(u)=\frac{\varphi_t\big((B_w+\alpha)^n u\big)}{w_{1+t}\cdots w_{n+t}}+\sum_{k=0}^{n-1} 
\frac{\alpha_{k+t}\big((B_w+\alpha)^{n-1-k}u\big)}{w_{k+1+t}\dots w_{n+t}}. \]
We apply this equality for $n=m$ and we use the third property of the map $\alpha$ to get
\begin{equation}\label{eq:sss4}
 \varphi_{m+t}(u)=\frac{\varphi_t\big((B_w+\alpha)^m u\big)}{w_{1+t}\cdots w_{m+t}}+\delta\sum_{k=0}^{m-1}\frac1{w_{k+1+t}\cdots w_{m+t}}.
\end{equation}
Since $\|\varphi\|_\infty\leq 1$ and using \eqref{eq:sss3}, we get
\[ \frac1{w_{1+t}\cdots w_{m+t}}\geq \delta\left(\frac{1+\delta}{\delta^2}\right)-1=\frac{1}{\delta}. \]
Multiplying \eqref{eq:sss4} by $w_{1+t}\cdots w_{m+t}$ we get
\[ \delta\sum_{k=0}^{m-1}w_{1+t}\cdots w_{k+t}=w_{t+1}\cdots w_{m+t}\varphi_{m+t}(u)-\varphi_t\big((B_w+\alpha)^m u\big) \]
so that
\[\sum_{k=0}^{m-1}w_{1+t}\cdots w_{k+t}\leq 1+\frac 1\delta. \]
This is true for an infinite number of positive integers $m$ which means that
\[\sup_{t\in\NN} \sum_{k=0}^{+\infty}w_{1+t}\cdots w_{k+t}\leq 1+\frac 1\delta. \]
Applying again Lemma \ref{lem:sss2}, but now to the sequence $x_k=w_k$, we get that
\[ \lim_{n\to+\infty}\sup_{k\in\NN} (w_k\cdots w_{k+n})^{1/n}<1, \]
which is \eqref{eq:sss1}.

We now finish like in the proof of Theorem 18 of \cite{BerMes20}. We briefly indicate the method.
Let $w'$ be defined by $w'_n=\frac1{w_{-n+1}}$ for all $n\in\ZZ$. Then $(B_w)^{-1}$ and $B_{w'}$ are conjugated
by an isometric isomorphism so that $B_{w'}$ is strongly structurally stable. Applying what we have done to $B_{w'}$
instead of $B_w$, we get that one of the following condition is satisfied:
\begin{equation}\label{eq:sss5}
\lim_{n\to+\infty}\inf_{k\in\NN} (w_{-k}\cdots w_{-k-n})^{1/n}>1
\end{equation}
\begin{equation}\label{eq:sss6}
\lim_{n\to+\infty}\sup_{k\in\NN} (w_{-k}\cdots w_{-k-n})^{1/n}<1.
\end{equation}
To conclude we observe that 
\begin{itemize}
 \item (A) corresponds to the case where \eqref{eq:sss1} and \eqref{eq:sss6} are both true;
 \item (B) corresponds to the case where \eqref{eq:sss2} and \eqref{eq:sss5} are both true;
 \item (C) corresponds to the case where \eqref{eq:sss2} and \eqref{eq:sss6} are both true;
 \item \eqref{eq:sss1} and \eqref{eq:sss5} cannot be simultaneously true by Lemma \ref{lem:sss3}, since $B_w$ is assumed to be strongly structurally stable.
\end{itemize}

\end{proof}

\section{Chaotic and frequently hypercyclic weighted shifts}


The proof of Theorem \ref{thm:kothe} is based on three tools. Firstly we need to exhibit a sequence of pairwise disjoint sets with positive lower density which are sufficiently sparse to produce frequently hypercyclic weighted shifts which are not chaotic. Up to our knowledge, this was used in all examples of a frequently hypercyclic yet not chaotic weighted shift (see \cite{BAYGRIPLMS} or \cite{BoGre18}).

\begin{lemma}\label{lem:sets}
There exist an increasing sequence $(N_k)_{k\geq 0}$ of nonnegative integers with $N_0=0$ and a sequence $(A_r)_{r\geq 1}$ of pairwise disjoint subsets of $\NN$ with positive lower density and satisfying the following properties:
\begin{enumerate}
\item[(a)] For $r\geq 1$, if $n\in A_r$, then for any $k\geq 0$, 
$$N_k\leq n<N_{k+1}\implies N_k+k\leq n<N_{k+1}-r.$$
\item[(b)] For $r,s\geq 1$, if $n\in A_r$, $m\in A_s$, $n>m$, then for any $k\geq 0$, 
$$N_k\leq n-m<N_{k+1}\implies N_k+\max(r,s)\leq n-m<N_{k+1}-\max(r,s).$$
\end{enumerate}
\end{lemma}

Our second tool is an abstract result to prove the frequent hypercyclicity of a shift on a Fr\'echet sequence space in which $(e_n)$ is an unconditional basis.

\begin{lemma}(\cite[Theorem 6.2]{BoGre18})\label{lem:critfhc}
Let $X$ be a Fréchet sequence space in which $(e_n)$ is an unconditional basis. Then a weighted shift on $X$ is frequently hypercyclic if and only if there exist a sequence $(\veps_r)_{r\geq 1}$ of positive numbers tending to zero and a sequence $(A_r)_{r\geq 1}$ of pairwise disjoint subsets of $\NN$ with positive lower density such that
\begin{enumerate}[(i)]
\item for any $r\geq 1$, 
\[ \sum_{n\in A_r}v_{n+r}e_{n+r}\textrm{ converges in }X;\]
\item for any $r,s\geq 1$, any $m\in A_s$ and any $j=0,\dots,r$,
\[ \left\|\sum_{n\in A_r,\ n>m} v_{n-m+j}e_{n-m+j}\right\|_s\leq \min(\veps_r,\veps_s), \]
where, for $n\geq 0$, 
\[ v_n=\frac{1}{w_1\cdots w_n}.\]
\end{enumerate}
\end{lemma}

The last preliminary result we need is a criterion for the existence of a chaotic weighted shift. Our counterexample will be a power series space of order $0$ and infinite type, namely a space 
\[\mathcal C_{0,\infty}(\alpha)=\left\{x=(x_n)_{n\geq 0}:\ \forall p\geq 1,\ p^{\alpha_n}|x_n|\xrightarrow{n\to+\infty}0\right\}\]
where $(\alpha_n)$ is a nondecreasing sequence of positive real numbers tending to infinity,
endowed with the family of seminorms 
\[ \|(x_n)\|_p=\sup_n |x_n|p^{\alpha_n}.\]
For such a space, we have the following characterization of the existence of a chaotic weighted shift.
\begin{lemma}(\cite[Proposition 4.2]{CGM21})\label{lem:existencechaotic}
\label{lem:chaoticws}
The space $\mathcal C_{0,\infty}(\alpha)$ supports a chaotic weighted shift if and only if 
\[\lim_{N\to+\infty}\frac{\sum_{n=0}^{N-1}\alpha_n}{\alpha_N}=+\infty.\]
\end{lemma}

\begin{proof}[Proof of Theorem \ref{thm:kothe}]
Let $(N_k)_{k\geq 0}$ and $(A_r)_{r\geq 1}$ be the two sequences given by Lemma \ref{lem:sets}. We define a nondecreasing sequence $(\alpha_n)_{n\geq 0}$ by $\alpha_0=1$ and 
\[
\left\{
\begin{array}{rcll}
\alpha_{N_k}&=&\alpha_0+\cdots+\alpha_{N_{k}-1}\\
\alpha_{n+1}&=&\alpha_n&\textrm{provided $n\neq N_{k}-1$ for some $k\in\NN$}.
\end{array}
\right.\]
We set $X=\mathcal C_{0,\infty}(\alpha)$.
Since $\sum_{n=0}^{N_{k}-1}\alpha_n/\alpha_{N_k}=1$, it follows from Lemma \ref{lem:existencechaotic} that $X$
does not support a chaotic weighted shift. Let us now prove that it admits a frequently hypercyclic weighted shift.
We set, for $n\geq 0$,
\[ \left\{
\begin{array}{rcll}
w_n&=&2^{\alpha_n}&\textrm{ provided $n\neq N_k$ for some $k\in\NN$}\\
w_n&=&\frac{2^{\alpha_{N_k}}}{2^{\alpha_{N_{k-1}}+\cdots+\alpha_{N_k-1}}}&\textrm{ provided }n=N_k.
\end{array}
\right.\]
In order to prove that $B_w$ is continuous, we just need to show that
\[ \forall p\geq 1,\ \exists q\geq 1,\ \exists C\geq 0,\ \forall n\geq 1,\ w_{n+1} p^{\alpha_n}\leq q^{\alpha_{n+1}}. \]
Now, for all values of $n$, we have
$$w_{n+1}p^{\alpha_n}\leq (2p)^{\alpha_n}\leq (2p)^{\alpha_{n+1}}$$
and we just need to choose $q=2p$.

Let us now show that $B_w$ is frequently hypercyclic. Observe that the definition of $w$ implies that 
\[ v_n=\left(\frac 12\right)^{\alpha_{N_k}+\cdots+\alpha_n}\]
where $k$ is the integer such that $N_k\leq n<N_{k+1}$. We shall apply Lemma \ref{lem:critfhc} with $\veps_r=2^{-r}r.$
\begin{enumerate}[(i)]
\item Let $r\geq 1$. We have to show that, for all $p\geq 1$, $p^{\alpha_{n+r}}v_{n+r}$ tends to $0$ as $n$ tends to $+\infty$, $n$ being in $A_r$. Let $k$ be such that $N_k\leq n<N_{k+1}$. Then $N_k+k\leq n<N_{k+1}-r$. Therefore
\begin{align*}
p^{\alpha_{n+r}}v_{n+r}&= p^{\alpha_{N_k}} \left(\frac 12\right)^{\alpha_{N_k}+\cdots+\alpha_n}\\
&\leq \left(p\left(\frac 12\right)^k\right)^{\alpha_{N_k}}
\end{align*}
and this goes to zero as $n$ (hence $k$) tends to $+\infty$.
\item Let $r,s\geq 1$, $m\in A_s$, $j=0,\dots,r$. We have to show that, for all $n>m$, $n\in A_r$, then 
\[ v_{n-m+j}s^{\alpha_{n-m+j}}\leq \min(\veps_r,\veps_s).\]
Let $k$ be such that $N_k\leq n-m<N_{k+1}$. Then
\[ N_k+\max(r,s)\leq n-m+j<N_{k+1} \]
so that
\begin{align*}
v_{n-m+j}s^{\alpha_{n-m+j}}&\leq \left(\left(\frac 12\right)^{\max(r,s)}s\right)^{\alpha_{N_k}}\\
&\leq \min(\veps_r,\veps_s).
\end{align*}
Hence, $B_w$ is a frequently hypercyclic weighted shift of the Köthe sequence space $X$ supporting no chaotic weighted shifts.
\end{enumerate}
\end{proof}

Acknowledgement: We thank Q. Menet and K-G. Grosse-Erdmann for useful discussions.

%


\providecommand{\bysame}{\leavevmode\hbox to3em{\hrulefill}\thinspace}

\providecommand{\MR}{\relax\ifhmode\unskip\space\fi MR }
\providecommand{\MRhref}[2]{%
  \href{http://www.ams.org/mathscinet-getitem?mr=#1}{#2}
}
\providecommand{\href}[2]{#2}

\end{document}